\documentclass[a4paper,12pt]{article}     
\usepackage{amsmath,amscd,color,amssymb}        
\usepackage{latexsym}                     
\usepackage[english, german]{babel}                

\usepackage{theorem}                 

\newtheorem{theorem}{Theorem}
\newtheorem{corollary}{Corollary}
\newtheorem{proposition}{Proposition}

\newenvironment{definition}
{\smallskip\noindent{\bf Definition\/}:}{\smallskip\par}

\newenvironment{remark}
{\smallskip\noindent{\bf Remark\/}.}{\smallskip\par}
\newenvironment{remarks}
{\smallskip\noindent{\bf Remarks\/}.}{\smallskip\par}

\newenvironment{proof}{\begin{ProofwCaption}{Proof}}{\end{ProofwCaption}}
\newenvironment{proof*}[1]{\begin{ProofwCaption}{{#1}}}{\end{ProofwCaption}}
\newenvironment{ProofwCaption}[1]%
  {\addvspace\theorempreskipamount \noindent{\it #1.}\rm}%
  {\qed \par \addvspace\theorempostskipamount}
\newcommand{\qedsymbol}{\mbox{$\Box$}}
\newcommand{\qed}{\hfill\qedsymbol}

\newcommand{\CC}{{\mathbb C}}

\newcommand{\ZZ}{{\mathbb Z}}

\newcommand{\Kring}[1]{K_0(\mbox{f.}{#1}\,\mbox{\rm -sets})}

\title{Equivariant Poincar\'e series and monodromy zeta functions of quasihomogeneous polynomials}
\author{Wolfgang Ebeling and Sabir M.~Gusein-Zade
\thanks{Partially supported by the DFG Mercator program (INST 187/490-1), the Russian government grant 11.G34.31.0005, RFBR--10-01-00678,
NSh--8462.2010.1.
Keywords: group actions, Burnside rings, zeta functions, Poincar\'e series.
AMS 2010 Math. Subject Classification: 32S40, 13D40, 19A22.
}
}
\date{}

\begin{document}
\selectlanguage{english}

\maketitle

\begin{abstract}
In earlier work, the authors described a relation between the Poincar\'e series and the classical monodromy zeta function corresponding to a quasihomogeneous polynomial. 
Here we formulate an equivariant version of this relation in terms of the Burnside rings of finite abelian groups and their analogues.
\end{abstract}


Let $f(x_1, \ldots, x_n)$ be a quasihomogeneous polynomial.
In \cite{E}, \cite{MRL}, there was described a relation between the Poincar\'e series $P_X(t)$ of the coordinate ring of a hypersurface singularity $X=\{ f=0 \}$
and the classical monodromy zeta function $\zeta_f(t)$  of the polynomial $f$. The relation involved 
the so called Saito duality: \cite{Saito1}, \cite{Saito2}. 
Namely, 
in \cite{MRL}, it was shown that
\begin{equation}\label{ini}
P_X(t)\cdot \mbox{Or}_X(t) = \widetilde{\zeta}_f^*(t)\,.
\end{equation}
Here $\mbox{Or}_X(t) $ is a rational function determined by the orbit types of the natural $\CC^\ast$-action on $X$ (see, e.g., \cite{Trieste}), $\widetilde{\zeta}_f=\zeta_f(t)/(1-t)$ is the reduced monodromy zeta function of $f$, and 
$\widetilde{\zeta}_f^*(t)$ is the Saito dual of $\widetilde{\zeta}_f(t)$
with respect to the quasidegree of the polynomial $f$.

This relation had no intrinsic explanation. It was obtained by computation of both sides and 
comparison of the results. In particular, the role of the Saito duality remained unclear. In \cite{G-Saito}, an equivariant version of the Saito duality
for a finite abelian group was formulated as a transformation between the Burnside rings of the group $G$ and of the group $G^*$ of its characters.
Here we use the Burnside rings and their analogues to define equivariant versions of the
ingredients of the relation (\ref{ini}) and to give an equivariant
analogue of it. This generalization can help to understand the role of the ingredients of the relation and, in particular, of the Saito duality in it.

Let $f$ be a quasihomogeneous polynomial in $n$ variables $x_1$, \dots, $x_n$ of degree $d$ with respect to the weights $q_1$, \dots, $q_n$
($q_i$ are positive integers, $\gcd(q_1, \ldots, q_n)=1$), i.e.\  
$$
f(\lambda^{q_1}x_1, \ldots, \lambda^{q_n}x_n)=\lambda^{d}f(x_1, \ldots, x_n)\mbox{\ \ for\ \ }\lambda\in\CC\,.
$$
From now on we assume that the polynomial $f$ determines the system of weights $(q_1,\ldots, q_n; d)$
in a unique way. This means that, in the lattice $\ZZ^n$ of monomials in 
$x_1$, \dots, $x_n$ (a point $(k_1, \ldots, k_n)\in\ZZ^n$ corresponds to the monomial $x_1^{k_1}\cdots x_n^{k_n}$), the monomials participating in $f$ with non-zero coefficients
generate an affine hyperplane (the hyperplane $\{\sum_i q_ik_i=d\}$). The system of weights $(q_1,\ldots, q_n; d)$ defines
a $\CC^*$-action on the space $\CC^n$:
\begin{equation}\label{action}
\lambda*(x_1, \ldots, x_n)=(\lambda^{q_1}x_1, \ldots, \lambda^{q_n}x_n)\,.
\end{equation}

Let
$$
G_f=\{\underline{\lambda}=(\lambda_1, \ldots, \lambda_n)\in(\CC)^n:f(\lambda_1 x_1, \ldots, \lambda_n x_n)=f(x_1, \ldots, x_n)\}
$$
be {\em the} (abelian) {\em symmetry group} of $f$, i.e.\  the group of diagonal linear transformations of $\CC^n$ preserving $f$. Let 
$$
\overline{G}_f=\{\underline{\lambda}\in(\CC)^n:f(\lambda_1 x_1, \ldots, \lambda_n x_n)=\alpha(\underline{\lambda})f(x_1, \ldots, x_n)\}
$$
be {\em the extended symmetry group} of $f$, i.e.\  the group of diagonal transformations of $\CC^n$ preserving $f$ up to a constant factor ($\alpha:\overline{G}_f\to\CC^*$
is a one-dimensional representation of the group $\overline{G}_f$). In other words, this is
the group of transformations preserving the hypersurface $X=\{f=0\}\subset\CC^n$.
The group $\overline{G}_f$ contains both the symmetry group $G_f$ and the group $\CC^*$ corresponding to the action~(\ref{action}) and is generated by these two subgroups.
The intersection of $G_f$ and $\CC^*$ is the cyclic subgroup of order $d$ in $\CC^*$. (It is generated by the monodromy transformation of $f$: see below.)

For a group $G$, let $R(G)$ be the ring of complex representations of $G$. As an abelian group,
$R(G)$ is freely generated by the isomorphism classes of 
irreducible representations of the group $G$. For an abelian group $G$ (say, for a subgroup of $G_f$
or of $\overline{G}_f$) all irreducible representations
are one-dimensional, i.e.\  are elements of the group of characters $G^*=\mbox{Hom}(G,\CC^*)$.

Let $\overline{G}$ be a subgroup of the extended symmetry group $\overline{G}_f$ of the polynomial $f$ containing the subgroup $\CC^*$, let $G=\overline{G}\cap G_f$.
We shall call an irreducible (one-dimensional) representation $\alpha$ of the group $\overline{G}$ non-positive ($\alpha\le 0$) or negative ($\alpha< 0$) 
if, for $a\in \CC^*\subset \overline{G}$, one has $\alpha(a)=a^k$ with $k$ non-positive or negative respectively.
Let $R_{-}(\overline{G})$ be the subring of the ring $R(\overline{G})$ of representations of the group $\overline{G}$ generated by non-positive representations.
The ring $R_{-}(\overline{G})$ contains the ideal $I$ generated by negative representations of $\overline{G}$.
Let $\widehat{R}_{-}(\overline{G})$ be the completion of the ring $R_{-}(\overline{G})$ with respect to the ideal $I$. Elements of the ring $\widehat{R}_{-}(\overline{G})$
are (formal) sums of the form $\sum\limits_{\alpha\in \overline{G}^*\!\!,\,\alpha\le 0} s_{\alpha}[\alpha]$ with integer coefficients $s_{\alpha}$. Let $\widehat{I}\subset \widehat{R}_{-}(\overline{G})$
be the corresponding completion of the ideal $I$.

One has the following homomorphism (an isomorphism) $\mbox{Exp}$ from $\widehat{I}$ regarded as
a group with respect to addition to $1 + \widehat{I}$ regarded 
as a group with respect to multiplication:
$$
\mbox{Exp}(\sum\limits_{\alpha\in \overline{G}^*\!\!,\,\alpha< 0} s_{\alpha}[\alpha])= \prod\limits_{\alpha\in \overline{G}^*\!\!,\,\alpha< 0}(1-[\alpha])^{-s_{\alpha}}.
$$ 
The inverse is the homomorphism $\mbox{Log}:1 + \widehat{I}\to \widehat{I}$ (cf. \cite{power}).

Let $A_X=\CC[x]/(f)$ be the coordinate ring of the zero level set $X=f^{-1}(0)$ of $f$.
The group $\overline{G}$ acts both on the ring 
$\CC[x]$ and on the ring $A_X$ by $a*g(x)=g(a^{-1}*x)$ for $a\in \overline{G}$.
Its representation on any $\overline{G}$-invariant one-dimensional subspace of $\CC[x]$ or of $A_X$ (say, on the subspace generated by a monomial) is non-positive and on any of them
except the one consisting of the constant functions is negative. Let $A$ be either $A_{\CC^n}=\CC[x]$ or $A_X$.
For a (one-dimensional) irreducible representation $\alpha$ of the group $\overline{G}$, let $A^{\alpha}$ be the corresponding
subspace of $A$: $A^{\alpha}=\{g\in A: a*g=\alpha(a)g \mbox{\ \ for\ \ } a\in \overline{G}\}$. For each $\alpha$, the subspace $A^{\alpha}$ is finite dimensional and is generated by monomials.

\begin{definition}
The $\overline{G}$-equivariant Poincar\'e series $P^{\overline{G}}$ of the ring $A$ is the element of the completion $\widehat{R}_{-}(\overline{G})$
of the representation ring of the group $\overline{G}$ defined by
\begin{equation}\label{Poincare}
P^{\overline{G}}=\sum\limits_{\alpha\in \overline{G}^*\!\!,\,\alpha\le 0} \dim A^{\alpha}\cdot [\alpha]\,.
\end{equation}
\end{definition}

\begin{remark}
For $\overline{G}=\CC^*$, an irreducible $\CC^*$-representation is a power of the tautological representation. If one denotes the inverse of the tautological
representation by $t$, one gets the Poincar\'e series (\ref{Poincare}) as a power series in $t$. In this case it coincides with the usual Poincar\'e series
of the ring $A$ corresponding to the (quasihomogeneous) grading defined by the weights. 
\end{remark}

For $A=A_X$ and for $A=A_{\CC^n}$ we shall denote the $\overline{G}$-equivariant Poincar\'e series $P^{\overline{G}}$ by $P_X^{\overline{G}}$ and $P_{\CC^n}^{\overline{G}}$ respectively.
Let $\alpha_{x_i}$, $i=1, \ldots, n$, and $\alpha_f$ be the representations of the group $\overline{G}$ on the one-dimensional subspaces in $\CC[x]$ generated by the functions $x_i$ and $f$ respectively.

\begin{proposition}
$$
P_X^{\overline{G}}=\frac{1-[\alpha_f]}{\prod_{i=1}^n (1-[\alpha_{x_i}])}\,.
$$
\end{proposition}

\begin{proof}
The proof is essentially the same as in the non-equivariant case. 
One has
$$
P_{\CC^n}^{\overline{G}}=
\sum_{(k_1,\ldots, k_n)\in\ZZ_{\ge 0}^n}[\alpha_{x_1}^{k_1}\cdots\alpha_{x_n}^{k^n}]=
\frac{1}{\prod_{i=1}^n (1-[\alpha_{x_i}])}\,.
$$
One considers the exact sequence
$$
0\longrightarrow A_{\CC^n} \stackrel{\cdot f}{\longrightarrow} A_{\CC^n} \stackrel{\pi}{\longrightarrow} A_X \longrightarrow 0\,.
$$
The homomorphism $\pi$ maps $A_{\CC^n}^{\alpha}$ to $A_X^{\alpha}$. If $g\in A_{\CC^n}^{\alpha}$, i.e.\  $a*g=\alpha(a)g$ for $a\in \overline{G}$,
then $a*(fg)=\alpha(a)\alpha_f(a)\cdot fg$, i.e.\   $fg\in A_{\CC^n}^{\alpha\cdot\alpha_f}$.
Therefore $\dim A_X^{\alpha}=\dim A_{\CC^n}^{\alpha}-\dim A_{\CC^n}^{\alpha/\alpha_f}$.
This yields the statement.
\end{proof}

\begin{corollary}
$$
\mbox{\rm Log\,}P_X^{\overline{G}}=\sum_{i=1}^n [\alpha_{x_i}]-[\alpha_f]\,.
$$
\end{corollary}

Note that $\mbox{Log\,}P_X^{\overline{G}}$ is an element of the representation ring 
$R(\overline{G})$ of the group $\overline{G}$ (more precisely of the subring $R_{-}(\overline{G})\subset R(\overline{G})$), not only of the completion $\widehat{R}_{-}(\overline{G})$.

Now we recall the necessary definitions and facts about Burnside rings of finite groups
(for more details see, e.g., \cite{G-Saito}) and give an appropriate extension of this notion
to subgroups of $\overline{G}_f$ containing $\CC^*$.

Let $G$ be a finite group. A $G$-set is a set with an action of the group $G$. A $G$-set is {\em irreducible} if the action
of $G$ on it is transitive. Isomorphism classes of irreducible $G$-sets are in one-to-one correspondence with
conjugacy classes of subgroups of $G$: to the conjugacy class containing a subgroup $H\subset G$ one associates the isomorphism class $[G/H]$
of the $G$-set $G/H$. The {\em Grothendieck ring} $\Kring{G}$ {\em of finite $G$-sets}
(also called the {\em Burnside ring} of $G$: see, e.g.,
\cite{Knutson}) is the (abelian) group generated by the isomorphism classes of finite $G$-sets
modulo the relation $[A\amalg B]=[A]+[B]$ for finite $G$-sets $A$ and $B$. The multiplication
in $\Kring{G}$ is defined by the 
cartesian product. As an abelian group, $\Kring{G}$
is freely generated by the isomorphism classes of irreducible $G$-sets. The element $1$ in
the ring $\Kring{G}$ is represented by the $G$-set consisting of one point (with the trivial $G$-action).

There is a natural homomorphism from the Burnside ring $\Kring{G}$ to the ring $R(G)$ of representations of the group $G$ which sends a $G$-set $X$ to the (vector)
space of (complex valued) functions on $X$. 

For a subgroup $H\subset G$ there are natural maps
$\mbox{Res}_{H}^{G}: \Kring{G}\to \Kring{H}$ and $\mbox{Ind}_{H}^{G}: \Kring{H}\to \Kring{G}$.
The {\em restriction map} $\mbox{Res}_{H}^{G}$ sends a $G$-set X to the same set considered with the $H$-action.
The {\em induction map} $\mbox{Ind}_{H}^{G}$ sends an $H$-set $X$ to the product $G\times X$ factorized
by the natural equivalence: $(g_1, x_1)\sim (g_2, x_2)$ if there exists $g\in H$ such that
$g_2=g_1g$, $x_2=g^{-1}x_1$ with the natural (left) $G$-action. The induction map $\mbox{Ind}_{H}^{G}$ sends the class $[H/H']$ ($H'$ is a subgroup of $H$) to the class $[G/H']$. 
Both maps are group homomorphisms, however the induction map $\mbox{Ind}_{H}^{G}$ is not a ring homomorphism. 

For an action of a group $G$ on a set $X$ and for a point $x\in X$, let $G_x=\{g\in G: gx=x\}$
be the isotropy group of the point $x$. For a subgroup $H\subset G$ let $X^{(H)}=\{x\in X: G_x=H\}$ be the set of points with the isotropy group $H$.

We recall the definition of the $G$-equivariant zeta function of $f$ from \cite{G-Saito}. See an explanation of this notion therein.

The monodromy transformation of $f$ can be defined as the element $h=h_f\in G_f$ given by
$$
h=\left(\exp(2\pi i \,q_1/d), \ldots, \exp(2\pi i \,q_n/d)\right)\,.
$$
As a map from the Milnor fibre $V_f=f^{-1}(1)$ of $f$ to itself,
$h$ defines an action (a faithful one) of the cyclic group $\ZZ_{d}=\langle h\rangle$
of order $d$ on $V_f$. Let
\begin{equation} \label{Defzeta}
\zeta_f(t)=\prod\limits_{q\ge 0} \left(\det(\mbox{id}-t\cdot h_{*}\mbox{\raisebox{-0.5ex}{$\vert$}}{}_{H_q(V_f)})\right)^{(-1)^q}
\end{equation}
be the (classical) monodromy zeta function of $f$ (that is the zeta function of the transformation $h$). One can show that in the described situation one has
$$
\zeta_f(t)=\prod\limits_{m\vert d}(1-t^m)^{s_m},
$$
where $s_m=\chi(V_f^{(\ZZ_{d/m})})/m$ are integers. If in (\ref{Defzeta}) one considers the action of $h_{*}$ on the reduced homology groups of $V_f$, one obtains the {\em reduced monodromy zeta function} $\widetilde{\zeta}_f(t) = \zeta_f(t)/(1-t)$.

There is a natural one-to-one correspondence between functions of the form
\begin{equation}\label{phi}
\varphi(t)=\prod\limits_{m\vert d}(1-t^m)^{s_m}
\end{equation}
and elements of the Burnside ring
$\Kring{\ZZ_{d}}$ of the cyclic group $\ZZ_d$, see \cite{G-Saito}.
The function $\varphi(t)$ from (\ref{phi}) corresponds to the element 
$\sum\limits_{m\vert \overline{d}}s_m[\ZZ_{d}/\ZZ_{d/m}] \in \Kring{\ZZ_{d}}$.
For the monodromy transformation $h_f$ in these terms one has 
\begin{equation}\label{zeta_Z_d}
\zeta_f=\sum_{H\subset \ZZ_d} \chi(V_f^{(H)}/\ZZ_d)[\ZZ_d/H]\in \Kring{\ZZ_d}\,.
\end{equation}
The coefficient $\chi(V_f^{(H)}/\ZZ_d)$ is the Euler characteristic of the space (a manifold) of orbits of type $\ZZ_d/H$ in $V_f$.

Now let $G$ be a subgroup of the symmetry group $G_f$ of the quasihomogeneous polynomial $f$ containing the monodromy transformation $h$.
Equation~(\ref{zeta_Z_d}) inspires the following definition (see \cite{G-Saito}).

\begin{definition}
The {\em $G$-equivariant zeta function} of $f$ is the element
\begin{equation}\label{G-zeta}
\zeta_f^{G}=\sum_{H\subset G} \chi(V_f^{(H)}/G)[G/H]
\end{equation}
of the Burnside ring $\Kring{G}$. 
\end{definition}

The coefficient $\chi(V_f^{(H)}/G)$ is the Euler characteristic of the space (a manifold) of orbits of type $G/H$ in $V_f$.

\begin{definition}
The {\em reduced $G$-equivariant zeta function} of $f$ is $\widetilde{\zeta}_f^G = \zeta_f^G-1$.
\end{definition}

Let $\Kring{\overline{G}}$ be the Grothendieck group of $\overline{G}$-sets with finite numbers of orbits and finite isotropy groups of points.
This group is freely generated by the classes of the $\overline{G}$-sets $\overline{G}/H$ with finite subgroups $H$.

\begin{remark}
There is no natural ring structure on $\Kring{\overline{G}}$.
\end{remark}

As above (for finite groups), one also has the natural induction map $\mbox{Ind}_{G}^{\overline{G}}: \Kring{G}\to \Kring{\overline{G}}$ which sends the class $[G/H]$ to the class $[\overline{G}/H]$
for a subgroup $H\subset G$. A left inverse to this map is the {\em reduction map} $\mbox{Red}:\Kring{\overline{G}}\to \Kring{G}$ which sends the class $[\overline{G}/H]$ to the class $[G/H\cap G]$
($H\subset \overline{G}$, $|H|<\infty$).

The group $\overline{G}$ acts on the zero level set $X$ of the function $f$.  The $\overline{G}$-equivariant orbit invariant $\mbox{Or}_X^{\overline{G}}$
counts the orbits of the $\overline{G}$-action on $X$ of different types.

\begin{definition} (cf. the definition in \cite{MRL}, \cite{Trieste})
$$
\mbox{Or}_X^{\overline{G}}:=\sum_{H\subset\overline{G},\, |H|<\infty}\chi(X^{(H)}/\overline{G})[\overline{G}/H]\in \Kring{\overline{G}}\,.
$$
\end{definition}

Let the {\em tautological map} $\mbox{Tau}$ from the ideal $I\subset R_{-}(\overline{G})$ to the group $\Kring{\overline{G}}$ be the (additive) group homomorphism mapping the class $[\alpha]$ of a one dimensional representation $\alpha$ to the class of the punctured space $\CC^*$ of the space (line) $\CC^1$ of the representation with the action of the group $\overline{G}$ defined by the representation $\alpha$. Note that $\mbox{Tau\,}[\alpha]$ can be defined in the same way for a positive representation $\alpha$ and one has $\mbox{Tau\,}[\alpha]=\mbox{Tau\,}[\alpha^{-1}]$ 
(via the isomorphism of $\CC^*$ with itself which sends $z\in \CC^*$ to $z^{-1}$).

\begin{theorem} One has
\begin{equation}\label{main}
\mbox{\rm Tau}({\mbox{\rm Log\,}}P_X^{\overline{G}})- \mbox{\rm Or}_X^{\overline{G}}=\mbox{\rm Ind}_G^{\overline{G}}\,\widetilde{\zeta}_f^G
\end{equation}
in $\Kring{\overline{G}}$.
\end{theorem}

\begin{proof}
For $I\subset I_0=\{1, \ldots, n\}$, let $|I|$ be the number of elements of $I$, let
$(\CC^*)^I:= \{(x_1, \ldots, x_n)\in \CC^n: x_i\ne 0 \mbox{ for }i\in I, x_i=0 \mbox{ for }i\notin I\}$
be the corresponding coordinate torus of dimension $I$ ($(\CC^*)^{\emptyset}=\emptyset$),
and let $G^I \subset \overline{G}$ be the isotropy subgroup $\{a\in \overline{G}: ax=x \mbox{ for } x\in (\CC^*)^I\}$ of points of $(\CC^*)^I$
(this isotropy subgroup is one and the same for all points $x\in (\CC^*)^I$).
Let $X^I=X\cap (\CC^*)^I$, $Y^I= X^I/\overline{G}$. One has
$$
\mbox{Or}_X^{\overline{G}}=\sum_I \chi(Y^I)\cdot [\overline{G}/G^I].
$$

The Milnor fibre $V_f=f^{-1}(1)$ is the union $\bigcup\limits_{I}(V_f \cap (\CC^*)^I)$ of $G$-invariant varieties. Therefore
$$
\widetilde{\zeta}_f^G=\sum\limits_{I}\chi((V_f \cap (\CC^*)^I)/G)[G/G^I]-1=
\sum\limits_{I}\left[\chi((\CC^*)^I/\overline{G})-\chi(Y^I)\right][G/G^I]-1\,.
$$
Note that, if $G^I\not\subset G$, then $V_f\cap(\CC^*)^I=\emptyset$.
For $|I|\ne 1$, one has $\chi((\CC^*)^I/\overline{G})=0$, for $|I|= 1$, $\chi((\CC^*)^I/\overline{G})=1$.
Therefore
$$
\mbox{Or}_X^{\overline{G}}+\mbox{Ind}_G^{\overline{G}}\,\widetilde{\zeta}_f^G=
\sum_{i=1}^n [\overline{G}/G^{\{i\}}]-\mbox{Ind}_G^{\overline{G}}\,1=
\mbox{Tau}(\mbox{Log\,}P_X^{\overline{G}}).
$$
\end{proof}

\begin{remarks}
{\bf 1.} One could prefer to have an equation like (\ref{main}) inside
the Burnside ring $\Kring{G}$. One can see that (\ref{main}) implies the equation
\begin{equation}\label{secondary}
\widetilde{\zeta}_f^G=\mbox{Red}(\mbox{Tau}(\mbox{Log\,}P_X^{\overline{G}})- \mbox{Or}_X^{\overline{G}})
\end{equation}
in $\Kring{G}$ with the reduction map $\mbox{Red}:\Kring{\overline{G}}\to\Kring{G}$.
However from a formal point of view (\ref{secondary}) is weaker than (\ref{main}).

{\bf 2.} Looking at the relation (\ref{main}), one observes that there is no Saito duality in the sense of \cite{G-Saito} involved in it. It appears in  (\ref{ini}) because of the method used to
encode the $\CC^*$-action on $X$ in $\mbox{Or}_X(t)$.

{\bf 3.} One can see that, generally speaking, both $\mbox{Or}_X^{\overline{G}}$ and
$\widetilde{\zeta}_f^G$ contain much more summands than $\mbox{Tau\,}(\mbox{Log\,}P_X^{\overline{G}})$.
In particular, $\mbox{Tau\,}(\mbox{Log\,}P_X^{\overline{G}})$ contains only summands represented by irreducible $\overline{G}$-sets isomorphic (as varieties) to $\CC^*$. This gives the hint that 
(\ref{main}) (and therefore also (\ref{ini})) is essentially a relation between $\mbox{Or}_X^{\overline{G}}$ and $\widetilde{\zeta}_f^G$, where the Poincar\'e series 
$P_X^{\overline{G}}$ plays rather the role of a correction term.
\end{remarks}


\bigskip
\noindent Leibniz Universit\"{a}t Hannover, Institut f\"{u}r Algebraische Geometrie,\\
Postfach 6009, D-30060 Hannover, Germany \\
E-mail: ebeling@math.uni-hannover.de\\

\medskip
\noindent Moscow State University, Faculty of Mechanics and Mathematics,\\
Moscow, GSP-1, 119991, Russia\\
E-mail: sabir@mccme.ru


\begin{thebibliography}{10}

\bibitem{E} W.~Ebeling: Poincar\'e series and monodromy of a two-dimensional quasihomogeneous
hypersurface singularity. Manuscripta Math. {\bf 107} (2002), no.3, 271--282.

\bibitem{MRL} W.~Ebeling, S.~M.~Gusein-Zade: Poincar\'{e} series and zeta function of the
monodromy of a quasihomogeneous singularity. Math. Res. Lett. {\bf 9} (2002), 509--513.

\bibitem{Trieste} W.~Ebeling, S.~M.~Gusein-Zade: Lectures on monodromy.
In: Singularities in geometry and topology, World Sci. Publ., Hackensack, NJ, 2007, pp.~234--252.

\bibitem{G-Saito} W.~Ebeling, S.~M.~Gusein-Zade:
Saito duality between Burnside rings for invertible polynomials. arXiv: 1105.1964.

\bibitem{power} S.~M.~Gusein-Zade, I.~Luengo, A.~Melle-Hernandez:
A power structure over the Grothendieck ring of varieties.
Math. Res. Lett. {\bf 11} (2004), no.1, 49--57. 

\bibitem{Knutson} D.~Knutson: $\lambda$-rings and the representation theory of the symmetric group.
Lecture Notes in Mathematics, Vol.~{\bf 308}, Springer-Verlag, Berlin, New York, 1973. 

\bibitem{Saito1} K.~Saito: Duality for regular systems of weights: a
pr\'{e}cis. In: Topological Field Theory, Primitive Forms and Related
Topics
(M.~Kashiwara, A.~Matsuo, K.~Saito, I.~Satake, eds.), Progress in Math.,
Vol.~{\bf 160}, Birkh\"auser, Boston, Basel, Berlin, 1998, pp.~379--426.

\bibitem{Saito2} K.~Saito: Duality for regular systems of weights. Asian J.
Math. {\bf 2} (1998), no.4, 983--1047.

\end{thebibliography}
\end{document}